\documentclass{article}
\usepackage[utf8]{inputenc}
\usepackage[margin=1in]{geometry}

\usepackage{amsmath}
\usepackage{amsthm}

\usepackage{charter}

\usepackage{hyperref}
\usepackage{amsmath}
\usepackage{amssymb}
\usepackage{mathtools}
\usepackage{enumerate,enumitem}
\usepackage{amsfonts,mathrsfs}
\usepackage{subcaption}
\DeclareMathAlphabet{\mathcal}{OMS}{zplm}{m}{n}
\usepackage{amsthm,thmtools}
\makeatletter
\def\th@plain{%
  \thm@notefont{}
  \itshape 
}
\def\th@definition{%
  \thm@notefont{}
  \normalfont 
}
\makeatother
\usepackage{color,soul}
\usepackage{algorithm, algpseudocode, algorithmicx}
\usepackage{comment}
\usepackage{authblk}

\newcommand*{\email}[1]{\href{mailto:#1}{\nolinkurl{#1}} } 

\newif\ifrelease
\releasefalse

\newcommand{\ropt}{r_{\rm opt}}
\DeclareMathOperator{\rank}{rank}
\newcommand{\twobytwo}[4]{\begin{bmatrix} #1 & #2 \\ #3 & #4 \end{bmatrix}}
\newcommand{\twobyone}[2]{\begin{bmatrix} #1 \\ #2 \end{bmatrix}}
\newcommand{\onebytwo}[2]{\begin{bmatrix} #1 & #2 \end{bmatrix}}
\newcommand{\threebyone}[3]{\begin{bmatrix} #1 \\ #2 \\ #3 \end{bmatrix}}
\newcommand{\onebythree}[3]{\begin{bmatrix} #1 & #2 & #3 \end{bmatrix}}
\newcommand{\threebythree}[9]{\begin{bmatrix} #1 & #2 & #3 \\ #4 & #5 & #6\\
  #7 & #8 & #9\end{bmatrix}}

\DeclareMathOperator{\Row}{Row}
\DeclareMathOperator{\Col}{Col}

\newcommand{\set}[1]{\mathcal{#1}}
\newcommand{\field}{\mathbb{K}}

\usepackage{kbordermatrix} 
\renewcommand{\kbldelim}{[} 
\renewcommand{\kbrdelim}{]}

\newcommand{\indexmat}[3]{#1(#2,#3)} 
\newcommand{\rinv}{\mathsf{R}}
\newcommand{\linv}{\mathsf{L}}
\DeclareMathOperator*{\argmin}{argmin}

\usepackage{tikz}
\usetikzlibrary {positioning}
\usepackage{tikz-cd}
\usetikzlibrary{matrix, calc, arrows}

\declaretheorem[name=Theorem]{theorem}
\declaretheorem[name=Proposition,numberlike=theorem]{proposition}

\declaretheorem[name=Lemma,numberlike=theorem]{lemma}
\declaretheorem[name=Corollary,numberlike=theorem]{corollary}

\theoremstyle{remark}

\theoremstyle{definition}

\newtheorem{construction}{Construction}

\title{Minimal Rank Completions for Overlapping Blocks}
\author[1,*]{Ethan N. Epperly}
\author[2]{Nithin Govindarajan}
\author[3]{Shivkumar Chandrasekaran}
\affil[1]{Division of Computing and Mathematical Sciences, California
Institute of Technology}
\affil[2]{Department of Electrical Engineering (ESAT), KU Leuven}
\affil[3]{Department of Electrical Engineering and Computer Engineering, University of
California, Santa Barbara}
\affil[*]{\textnormal{Corresponding author. Email: \email{eepperly@caltech.edu}}}
\date{\today}

\begin{document}

\maketitle

\begin{abstract}
  We consider the multi-objective optimization problem of choosing the bottom left block-entry of a block lower triangular matrix to minimize the ranks of all block
  sub-matrices. We provide a proof that there exists a simultaneous rank-minimizer by constructing the complete set of all minimizers.
\end{abstract}

\let\thefootnote\relax\footnotetext{\textbf{Keywords:} Matrix Completion, Low-rank Structure, Minimal Rank Completion}

\let\thefootnote\relax\footnotetext{\textbf{AMS Subject Classifications:} 15A83, 15A29, 65F55}

\section{Introduction}

This article considers the following problem: given a block-triangular array of the form
\begin{equation} \label{eq:lower_triangular_array}
    \begin{bmatrix}
    A_{11} \\
    A_{21} & A_{22} \\
    \vdots & \vdots & \ddots \\
    X & A_{n2} & \cdots & A_{nn}
    \end{bmatrix}
\end{equation}
can the $(n,1)$-block $X$ be chosen to simultaneously minimize
the ranks of the subblocks
\begin{equation}\label{eq:hankel_blocks}
    \begin{bmatrix} A_{11} \\ A_{21} \\ \vdots \\
    X\end{bmatrix}, \begin{bmatrix}
    A_{21} & A_{22} \\ \vdots & \vdots \\ X & A_{n2}
    \end{bmatrix},\ldots,
    \begin{bmatrix}
    X & A_{n2} & \cdots & A_{nn}
    \end{bmatrix}?
\end{equation}
The main theorem of this paper answers this question in the affirmative.

\begin{theorem}\label{thm:main_theorem}
  There exists an $X$ simultaneously minimizing the ranks of each of the matrices in \eqref{eq:hankel_blocks}.
\end{theorem}

Our proof is constructive and can be implemented as an algorithm to recover all such $X$. 

\paragraph{Related work.} The problem of determining the minimal rank completion of a matrix with partially specified entries has received considerable interest
in the literature \cite{candes_matrix_2009}, and has seen applications in collaborative filtering \cite{goldberg_using_1992,rennie_fast_2005}, system identification
\cite{liu_interior-point_2009}, and remote sensing \cite{schmidt_multiple_1986}. For an arbitrary distribution of the missing entries, minimal rank completion problems
are NP hard \cite{buss_computational_1999}, but under certain conditions they can be solved exactly using semidefinite programming \cite{candes_matrix_2009}.

An alternate line of work \cite{kaashoek_unique_1988,woerdeman1989minimal,woerdeman_minimal_1993,eidelman2014separable} has shown that the entire solution set of a minimal
rank completion problem can be computed exactly using matrix factorizations provided the location of the missing entries have a structured pattern. For example, the
complete solution set of the minimal rank completion problem can be computed when the missing entries are arranged in a block triangular form
\cite{kaashoek_unique_1988,woerdeman1989minimal} or in a banded structure \cite{woerdeman_minimal_1993}. These methods may be more computationally tractable than
semidefinite programming-based approaches and they have the advantage that they provide the complete solution set. In particular, they can certify or dis-certify
uniqueness of the rank-minimizing choice. However, the requirement for the missing entries to satisfy a certain pattern can significantly limit the range of applicability
of these techniques.

We are unaware of any existing work on the overlapping block minimal rank completion problem considered in this paper. The closest related work we are aware of is the work of \cite{kaashoek_unique_1988, woerdeman1989minimal} on block triangular minimal rank completion problem, but this work considers the case when the \textit{missing entries} comprise a triangular matrix. By contrast, in our case only a single block is missing, but we are interested in minimizing the rank of not one, but a collection of blocks with overlapping entries. Simultaneous matrix completion problems over finite fields have been studied in theoretical computer science \cite{harvey_complexity_2006}, but the finite size of the field plays a critical role in this theory. By contrast, this work makes no assumptions on the cardinality of the field. 

\paragraph{Motivation and potential applications.} We encountered this type of minimal rank completion problem naturally in the construction of minimal representations for
rank-structured matrices \cite{CEG19}. Consider a matrix $M$ for which \eqref{eq:lower_triangular_array} denotes its strictly block lower triangular part.
The sequentially semiseparable (SSS) representation \cite{chandrasekaran2005some} of $M$ compresses the matrix $M$ by storing only low-rank approximations of each of the
overlapping \emph{Hankel blocks}
given by \eqref{eq:hankel_blocks}. Rather than storing these low-rank approximations independently, the SSS representation leverages the overlap of the blocks
to achieve further levels of compression. The strictly upper triangular part of $M$ is compressed similarly.

The SSS representation is effective at representing matrices with approximately low-rank off-diagonal blocks, but it can faulter for matrices possessing more complicated
types of low-rank structure. In an effort to develop more robust types of representations,
we considered in \cite{CEG19} a variant of SSS representation which we called the \emph{cycle
  semiseparable} (CSS) representation. This representation introduces the additional flexibility of perturbing the bottom left entry of $M$ before the lower triangular
array \eqref{eq:lower_triangular_array} is compressed using the SSS representation. Since the size of the SSS representation scales with the rank of the Hankel blocks
\eqref{eq:hankel_blocks}, it is advantageous to choose this corner perturbation to minimize the rank of each of the Hankel blocks. Initially, one
might (and we did) suspect that this problem would involve some trade-offs: minimizing the rank of one block might necessarily increase the rank of another.
Surprisingly,
Theorem~\ref{thm:main_theorem} shows that this is not the case: there exists a single choice minimizing the ranks of all blocks.
We speculate that the overlapping
minimal rank completion problem (or generalizations of it) may be of use in constructing minimal representations for other SSS-related formats and
may have other applications in the study of rank-structured matrices, such as determining the nearest SSS matrix with generators of rank $r$ to a given matrix.

\paragraph{Outline.} We begin with a ``warm-up'' by summarizing the solution to the block $2\times 2$ minimal rank completion problem, as derived in
\cite{kaashoek_unique_1988,woerdeman1989minimal}, in Section~\ref{sec:block_2x2}. We then extend this technique to handle our general problem in Section~\ref{sec:main_proof},
providing a proof of Theorem~\ref{thm:main_theorem}. Our construction completely characterizes the solution set, which is an affine space whose dimension
we will identify. We will discuss computational issues and conclude in Section~\ref{sec:discussion_and_conclusions}.

\paragraph{Notation.} Sets of row or column indices will be denoted by calligraphic letters and entries of a matrix $A$ occurring in rows indexed by $\set{I}$
and columns indexed by $\set{J}$ will be denoted as $A(\set{I},\set{J})$.
The set of rows of $A$ indexed by $\set{I}$ or columns indexed by
$\set{J}$ will be denoted $\indexmat{A}{\set{I}}{:}$
and $\indexmat{A}{:}{\set{J}}$, respectively. 

When notating a block matrix $X$ corresponding to, say, row indexes $\overline{\set{I}}$ and $\set{I}$
and column indexes $\overline{\set{J}}$ and $\set{J}$, we shall write
\begin{equation*}
    X = \kbordermatrix{
        & \overline{\set{J}}              & \set{J}              \\
        \set{I}                         & \indexmat{X}{\set{I}}{\overline{\set{J}}}  & \indexmat{X}{\set{I}}{\set{J}} \\
        \overline{\set{I}} & \indexmat{X}{\overline{\set{I}}}{\overline{\set{J}}} &
        \indexmat{X}{\overline{\set{I}}}{\set{J}}
      }.
\end{equation*}
This expression is purely a notation: we are not implying that the rows in $\set{I}$ precede
those in $\overline{\set{I}}$ in the natural ordering. Rather, we encourage the reader to think
of the matrix as partitioned in this way, despite any complicated interleavings between the 
rows and columns indexed by these sets.

The row and column spaces of a matrix $M$ will be denoted $\Row M$ and $\Col M$ respectively. Matrices will be permitted to have zero rows or columns.
Our main result will hold for matrices taking values in an arbitrary
field which we will denote $\field$. Readers unconcerned with such levels of generality are free to consider the field $\field$ to consist of the real
or complex numbers.

\section{The Block \texorpdfstring{$2\times 2$}{Two-by-two} Minimal Rank Completion
Problem}\label{sec:block_2x2}

As a warm-up, let us consider the block $2\times 2$ minimal rank completion problem:
\begin{equation}
  \mbox{find } X \mbox{ such that } \rank \twobytwo{B}{C}{X}{D} \mbox{ is minimal}, \label{eq:2x2_lrcp}
\end{equation}
where $B$, $C$, and $D$ are provided matrices of the appropriate size. 
We shall summarize the complete solution to this problem derived originally, in a more general setting, by Kaashoek and Woerdeman \cite{kaashoek_unique_1988,woerdeman1989minimal}.
An alternate derivation is given by \cite{eidelman2014separable}, though they do not characterize the
complete solution set (nor do they claim to). Our presentation in this section
will be brief and mainly serves to set up tools and foreshadow techniques needed
to prove Theorem~\ref{thm:main_theorem}.

We have the following lower bound for this minimal rank completion problem, originally due to \cite{woerdeman1987lower}.

\begin{proposition}
  For any matrix $X$ of the appropriate size,
  \begin{equation} \label{eq:2x2_lower_bound}
    \rank \twobytwo{B}{C}{X}{D} \ge \rank \onebytwo{B}{C} + \rank \twobyone{C}{D} - \rank C =: \ropt
  \end{equation}
\end{proposition}

The proof is quite straightforward and basically resolves to noting that, regardless of the choice of $X$, $\onebytwo{X}{D}$ contains
at least $\rank \twobyone{C}{D} - \rank C$ rows linearly independent of $\onebytwo{B}{C}$. In fact, this lower bound is achieved, as
we shall soon show.
Let us first start off with a special case, for which there is a unique solution furnished by the following constructive lemma. This is a particular special case of the full characterization of unique completions for the $2\times 2$ block minimal rank completion problem \eqref{eq:2x2_lrcp} from \cite{kaashoek_unique_1988}.

\begin{lemma}[Unique Completion Lemma] \label{lem:ucl}
  Consider an instance of the block $2\times 2$ minimal rank completion problem \eqref{eq:2x2_lrcp} for which can be block-partitioned as 
  \begin{equation*}
    \twobytwo{B}{C}{X}{D} = \left[\begin{array}{c|cc} B_{1} & C_{11} & C_{12} \\
    B_2 & C_{21} & C_{22} \\\hline 
    X & D_1 & D_2
    \end{array}\right]
  \end{equation*}
  where
  \begin{enumerate}[label=(\arabic*)]
  \item $\Col \begin{bmatrix} B_1 & C_{11} & C_{12} \end{bmatrix} = \Col \begin{bmatrix} C_{11} & C_{12} \end{bmatrix}$,
  \item $\Row \begin{bmatrix} C_{12} \\ C_{22} \\ D_2 \end{bmatrix} = \Row \twobyone{C_{12}}{C_{22}}$,
  \item $\Col C_{11} \cap \Col C_{12} = \{0\}$, 
  \item $\Row C_{22} \cap \Row C_{12} = \{0\}$,
  \item $C_{11}$ has full column rank, and
  \item $C_{22}$ has full row rank. 
  \end{enumerate}
  Then \eqref{eq:2x2_lrcp} has a unique solution $X$. Moreover, this solution satisfies the properties
  \begin{enumerate}[label=(\Roman*)]
  \item $\Row \onebytwo{X}{D} \subseteq \Row \onebytwo{B}{C}$, \label{item:ucl_row}
  \item $\Col \twobyone{B}{X}\subseteq \Col \twobyone{C}{D}$, and \label{item:ucl_col}
  \item $X$ has an affine dependence on $C_{21}$, $B_2$, and $D_1$. \label{item:ucl_affine}
  \end{enumerate}
\end{lemma}

\begin{proof}
  Without loss of generality, we are free to assume that $\onebythree{B_1}{C_{11}}{C_{12}}$
  and $\threebyone{C_{12}}{C_{22}}{D_2}$ have full row and column rank, respectively,
  by replacing each with a maximal linearly independent subcollection of rows or columns. One can confirm that this modification
  preserves the solution set of \eqref{eq:2x2_lrcp}, as well as the
  hypotheses.

  By the first two hypotheses, we thus have that $\onebytwo{C_{11}}{C_{12}}$ and $\twobyone{C_{12}}{C_{22}}$ have full row and column rank
  respectively. But then by the next four hypotheses, we have that $C$ has linearly
  independent columns and rows respectively, so $C$ is square and invertible.
  Then $X = DC^{-1}B$ is the unique rank-minimizing $X$ because 
  \begin{equation*}
    \rank \twobytwo{B}{C}{X}{D} = \rank \twobytwo{B}{C}{X - DC^{-1} B}{0} = \ropt = \rank \twobyone{B}{C} \mbox{ if, and only if, } X = DC^{-1}B,
  \end{equation*}
  where we used the fact that the rank of a matrix is preserved by adding some matrix-weighted multiple of
  one block row to another.

  Finally, we establish the affine dependence of $X$. For a matrix $M$ with full column or row rank, denote by $M^\linv$ or $M^\rinv$ any
  distinguished left or right inverse. Under the standing hypotheses, it is straightforward to derive
  (see Appendix~\ref{app:Cinvder})
  that $C^{-1}$ possesses the block structure
  \begin{equation*}
    C^{-1} = \twobytwo{\onebytwo{I}{0} \onebytwo{C_{11}}{C_{12}}^\rinv}{0}{\left(\onebytwo{0}{I} - \twobyone{C_{12}}{C_{22}}^\linv \twobyone{0}{I}
        \onebytwo{C_{21}}{C_{22}}\right) \onebytwo{C_{11}}{C_{12}}^\rinv}{\twobyone{ C_{12}}{C_{22}}^\linv \twobyone{0}{I}},
  \end{equation*}
  where the partitioning is such that the products $CC^{-1}$ and $C^{-1}C$ are conformal.
  With this structure clear, $X$ can be written as the following affine function of 
  $B_2$, $C_{21}$, and $D_1$:
  \begin{equation*}
    X =  D_1 E + F C_{21} G +  H B_2 + K,
  \end{equation*}
  where
  \begin{align*}
      E & = \onebytwo{I}{0} \onebytwo{C_{11}}{C_{12}}^\rinv B_{1},&
      F & = -  D_2 \twobyone{C_{12}}{C_{22}}^\linv \twobyone{0}{I}, \\
      H &= D_2 \twobyone{ C_{12}}{C_{22}}^\linv \twobyone{0}{I},&
      K &= D_2  \onebytwo{0}{I}  \onebytwo{C_{11}}{C_{12}}^\rinv B_{1} + FC_{22}K,
  \end{align*}
  with $\onebytwo{C_{11}}{C_{12}}^\rinv B_{1}$ partitioned as
  $\twobyone{G}{E}$ such that $\onebytwo{C_{11}}{C_{12}}\twobyone{G}{E}$ is conformal.
\end{proof}

\begin{construction}[Solution of block $2\times 2$ minimal rank completion problem \eqref{eq:2x2_lrcp}]\label{cons:2x2_lrcp}
  Consider the block $2\times 2$ minimal rank completion problem \eqref{eq:2x2_lrcp}. Do the following:

  \begin{enumerate}
  \item Choose a minimal set of columns $\set{J}$ of $B$ such that $\Col \onebytwo{\indexmat{B}{:}{\set{J}}}{C} = \Col \onebytwo{B}{C}$. 
    Extend this by a minimal additional set of columns $\set{J}'$, disjoint from $\set{J}$, such that $\Col \indexmat{B}{:}{\set{J}\cup\set{J}'} = \Col B$.
    Denote $\overline{\set{J}}$ to be the remaining columns of $B$ not in $\set{J}\cup\set{J}'$.\label{step:2x2_lrcp_cols}  
  \item Similarly, choose a minimal set of rows $\set{I}$ of $D$ such that $\Row \twobyone{C}{\indexmat{D}{\set{I}}{:}} = \Row \twobyone{C}{D}$.
    Extend this by a minimal additional set of rows $\set{I}'$, disjoint from $\set{I}$, such that $\Row \indexmat{D}{\set{I}\cup \set{I}'}{:} = \Row D$.
    Denote $\overline{\set{I}}$ to be the remaining rows of $D$ not in $\set{I}\cup\set{I}'$. 
    \label{step:2x2_lrcp_rows}
  \item Consider $X$ as a block matrix corresponding to these index sets:
    \begin{equation*}
      X = \kbordermatrix{
        & \overline{\set{J}}              & \set{J}'               & \set{J}              \\
        \set{I}                         & \indexmat{X}{\set{I}}{\overline{\set{J}}}  & \indexmat{X}{\set{I}}{\set{J}'}   &  \indexmat{X}{\set{I}}{\set{J}} \\
        \set{I}'                        & \indexmat{X}{\set{I}'}{\overline{\set{J}}} & \indexmat{X}{\set{I}'}{\set{J}'}  &  \indexmat{X}{\set{I}'}{\set{J}} \\
        \overline{\set{I}} & \indexmat{X}{\overline{\set{I}}}{\overline{\set{J}}} &
        \indexmat{X}{\overline{\set{I}}}{\set{J}'} & \indexmat{X}{\overline{\set{I}}}{\set{J}} \\
      }.
    \end{equation*}
    Choose the strictly block upper triangular part of this array with
    respect to this partitioning---namely $\indexmat{X}{\set{I}}{\set{J}'}$,
    $\indexmat{X}{\set{I}}{\set{J}}$,
    and $\indexmat{X}{\set{I}'}{\set{J}}$---arbitrarily. 

  \item For the rank of the completed matrix to be minimal, the remaining entries are now fixed as follows:

    \begin{enumerate}
    \item Let $\indexmat{X}{\set{I}'}{\set{J}'}$ be the unique rank-minimizing solution to the minimal rank completion problem 
      \begin{equation} \label{eq:2x2_filling}
        \indexmat{X}{\set{I}'}{\set{J}'} := \argmin_Y \rank
        \threebythree{\indexmat{B}{:}{\set{J}'}}{\indexmat{B}{:}{\set{J}}}{C}
        {\indexmat{X}{\set{I}}{\set{J}'}}{\indexmat{X}{\set{I}}{\set{J}}}{\indexmat{D}{\set{I}}{:}}
        {Y}{\indexmat{X}{\set{I}'}{\set{J}}}{\indexmat{D}{\set{I}'}{:}},
      \end{equation} \label{item:filling1}
      as provided by the unique completion lemma.

    \item Observe the columns of $\indexmat{B}{:}{\overline{\set{J}}}$ are spanned by the linearly independent columns of $\indexmat{B}{:}{\set{J}\cup\set{J}'}$
      and the rows of $\indexmat{D}{:}{\overline{\set{I}}}$ are spanned by the linearly independent rows of $\indexmat{D}{:}{\set{I}\cup\set{I}'}$.
      Therefore, there exist (unique) matrices $R$ and $Q$ such that
      \begin{equation*}
        \indexmat{B}{:}{\set{J}\cup\set{J}'}Q = \indexmat{B}{:}{\overline{\set{J}}}, \quad \indexmat{D}{\set{I}\cup\set{I}'}{:} =
        R\indexmat{D}{\overline{\set{I}}}{:}.
      \end{equation*}
      Set
      \begin{gather*}
        \indexmat{X}{\set{I}\cup\set{I}'}{\overline{\set{J}}} :=  \indexmat{X}{\set{I}\cup\set{I}'}{\set{J}\cup\set{J}'}Q, \quad 
        \indexmat{X}{\overline{\set{I}}}{\set{J}\cup\set{J}'} := R\indexmat{X}{\set{I}\cup\set{I}'}{\set{J}\cup\set{J}'},\\
        \indexmat{X}{\overline{\set{I}}}{\overline{\set{J}}} := R \indexmat{X}{\set{I}\cup\set{I}'}{\set{J}\cup\set{J}'}Q.
      \end{gather*} \label{item:filling2}
    \end{enumerate}
  \end{enumerate}
\end{construction}

\begin{theorem}
  Construction~\ref{cons:2x2_lrcp} produces the complete solution set to \eqref{eq:2x2_lrcp}. Moreover, the solution set is an affine space with dimension
  \begin{gather*}
    \left( \rank \twobyone{C}{D} - \rank C \right) \cdot \rank B  + \rank D \cdot \left( \rank \onebytwo{B}{C} - \rank C \right)\\
    - \left( \rank \twobyone{C}{D} - \rank C \right) \cdot \left( \rank \onebytwo{B}{C} - \rank C \right).
  \end{gather*}
\end{theorem}

The dimension formula is a special case of the more general formula \cite[Eq.~(0.1)]{woerdeman1989minimal} characterizing the size of the solution set for the block triangular minimal rank completion problem.

\begin{proof}
  First, we observe that the hypotheses of the unique completion lemma are satisfied in
  step~\ref{item:filling1}, so the construction is valid. We shall prove that every choice of the free variables in Construction~\ref{cons:2x2_lrcp} leads to a matrix $X$
  which saturates the lower bound \eqref{eq:2x2_lower_bound}, hence proving every output of Construction~\ref{cons:2x2_lrcp} is a solution to \eqref{eq:2x2_lrcp}.

  Let $X$ be an output of Construction~\ref{cons:2x2_lrcp}. Let $\set{K}$ denote the indices of the rows of $C$ and $\set{L}$ denote the
  indices of the columns of $C$. Then, with respect to the partitioning induced by the index sets, the completed matrix can be written
  in block form as 
  \begin{equation*}
    \twobytwo{B}{C}{X}{D} = 
    \kbordermatrix{
      & \overline{\set{J}}              & \set{J}'               & \set{J}              & \set{L} \\
      \set{K} & \indexmat{B}{:}{\overline{\set{J}}} & \indexmat{B}{:}{\set{J}'} & \indexmat{B}{:}{\set{J}} & C \\
      \set{I}                         & \indexmat{X}{\set{I}}{\overline{\set{J}}}  & \indexmat{X}{\set{I}}{\set{J}'}   &  \indexmat{X}{\set{I}}{\set{J}}
      & \indexmat{D}{\set{I}}{:} \\
      \set{I}'                        & \indexmat{X}{\set{I}'}{\overline{\set{J}}} & \indexmat{X}{\set{I}'}{\set{J}'}  &  \indexmat{X}{\set{I}'}{\set{J}}
      & \indexmat{D}{\set{I}'}{:} \\
        \overline{\set{I}} & \indexmat{X}{\overline{\set{I}}}{\overline{\set{J}}} &
        \indexmat{X}{\overline{\set{I}}}{\set{J}'} & \indexmat{X}{\overline{\set{I}}}{\set{J}} &
        \indexmat{D}{\overline{\set{I}}}{:} \\
    }.
  \end{equation*}
  By step~\ref{item:filling2} of the construction, the columns corresponding to indices in
  $\overline{\set{J}}$ are linear combinations of the remaining columns and are thus immaterial
  to the rank. Once one disregards these columns, the rows indexed by $\set{I}'\cup\overline{\set{I}}$ are linear combinations of the previous
  rows and can similarly be ignored by steps~\ref{item:filling1} of the 
  construction and optimality property
  \ref{item:ucl_row} of the unique completion lemma. Thus,
  \begin{equation*}
    \rank \twobytwo{B}{C}{X}{D} = \rank \begin{bmatrix}
      \indexmat{B}{:}{\set{J}'} & \indexmat{B}{:}{\set{J}} & C \\
      \indexmat{X}{\set{I}}{\set{J}'}   &  \indexmat{X}{\set{I}}{\set{J}} & \indexmat{D}{\set{I}}{:}
      \end{bmatrix} = \ropt,
  \end{equation*}
  achieving the lower bound in \eqref{eq:2x2_lower_bound} and proving that $X$ is indeed a minimal rank completion. The solution $X$ depends affinely on the freely chosen portions of itself because of the affine dependence property \ref{item:ucl_affine} provided
  by the unique completion lemma. The dimension of this space is clear by counting the number of entries in the freely chosen
  portions of $X$ by a inclusion-exlusion argument.

  Now we must show that Construction~\ref{cons:2x2_lrcp} produces \emph{all} solutions to \eqref{eq:2x2_lrcp}. To do this, we shall show that
  every matrix completed with an $X$ \emph{not} produced by Construction~\ref{cons:2x2_lrcp} possesses strictly higher rank than $\ropt$.
  Suppose that $\indexmat{X}{\set{I}'}{\set{J}'}$ is not constructed as in step~\ref{item:filling1}. Then,
  since $\indexmat{X}{\set{I}'}{\set{J}'}$ is the unique rank minimizer for \eqref{eq:2x2_filling}, we see that  
  \begin{align*}
    \rank \twobytwo{B}{C}{X}{D} 
    \ge \rank \begin{bmatrix}
      \indexmat{B}{:}{\set{J}'} & \indexmat{B}{:}{\set{J}} & C \\
      \indexmat{X}{\set{I}}{\set{J}'}   &  \indexmat{X}{\set{I}}{\set{J}} & \indexmat{D}{\set{I}}{:} \\
      \indexmat{X}{\set{I}'}{\set{J}'} & \indexmat{X}{\set{I}'}{\set{J}} & \indexmat{D}{\set{I}'}{:}
    \end{bmatrix}
    > \rank \begin{bmatrix}
      \indexmat{B}{:}{\set{J}'} & \indexmat{B}{:}{\set{J}} & C \\
      \indexmat{X}{\set{I}}{\set{J}'}   &  \indexmat{X}{\set{I}}{\set{J}} & \indexmat{D}{\set{I}}{:} \\
    \end{bmatrix} = \ropt,
  \end{align*}
  where in the first equality we add a weighted linear combination of the first two block rows to the third. Thus, for $X$ to be minimal,
  $\indexmat{X}{\set{I}'}{\set{J}'}$ must be set as in step~\ref{item:filling1}. The proof that the assignments in step~\ref{item:filling2}
  are necessary uses a similar idea and is omitted.
  We conclude that Construction~\ref{cons:2x2_lrcp} completely characterizes the set of solution of \eqref{eq:2x2_lrcp}.
\end{proof}

\section{Solution of the General Problem} \label{sec:main_proof}

Recall we are interested in the following simultaneous optimization problem:
\begin{equation}
  \label{eq:overlapping_block_mrcp}
  \mbox{find } X \mbox{ such that the rank of each block in \eqref{eq:hankel_blocks} is minimal}.
\end{equation}
 We shall show that the following construction produces all solutions.

\begin{construction}[Solution of the overlapping block minimal rank completion problem] \label{cons:overlapping_block_mrcp}
  Consider the overlapping block minimal rank completion problem \eqref{eq:overlapping_block_mrcp}. Do the following

  \begin{enumerate}
  \item Recursively construct column index sets
    \begin{equation*}
      \emptyset = \set{L}_n \subseteq \set{L}_{n-1} \subseteq \cdots \subseteq \set{L}_1 \subseteq
      \set{L}_0
    \end{equation*}
    where, for $i \in \{1,2,\ldots,n-1\}$, $\set{L}_i$ is a minimal set of indices containing $\set{L}_{i+1}$ such that
    \begin{equation*}
      \Col \begin{bmatrix}
        A_{i1}(:,\set{L}_i) & A_{i2} & \cdots & A_{ii} \\
        \vdots                  & \vdots     & \ddots & \vdots         \\
        A_{(n-1)1}(:,\set{L}_i) & A_{(n-1)2} & \cdots & A_{(n-1)i}
      \end{bmatrix} =
      \Col \begin{bmatrix}
        A_{i1} & A_{i2} & \cdots & A_{ii} \\
        \vdots     & \vdots     & \ddots & \vdots         \\
        A_{(n-1)1} & A_{(n-1)2} & \cdots & A_{(n-1)i}
      \end{bmatrix}
    \end{equation*}
    and $\set{L}_0$ consists of all the columns of $X$.
  
  \item Recursively construct row index sets sets
    \begin{equation*}
      \emptyset = \set{K}_0 \subseteq \set{K}_1 \subseteq \cdots \subseteq \set{K}_{n-1}
      \subseteq \set{K}_n
    \end{equation*}
    such that, for $i \in \{1,2,\ldots,n-1\}$, $\set{K}_i$ is a minimal set of indices containing $\set{K}_{i-1}$ for which
    \begin{equation*}
      \Row \begin{bmatrix}
        A_{(i+1)2}          & \cdots & A_{(i+1)(i+1)} \\
        \vdots              & \ddots & \vdots \\
        A_{(n-1)2}          & \cdots & A_{(n-1)(i+1)} \\
        A_{n2}(\set{K}_i,:) & \cdots & A_{n(i+1)}(\set{K}_i,:)
      \end{bmatrix}
      =
      \Row \begin{bmatrix}
        A_{(i+1)2}          & \cdots & A_{(i+1)(i+1)} \\
        \vdots              & \ddots & \vdots \\
        A_{(n-1)2}          & \cdots & A_{(n-1)(i+1)} \\
        A_{n2}              & \cdots & A_{n(i+1)}
      \end{bmatrix}
    \end{equation*}
    and $\set{K}_n$ consists of all the rows of $X$. Note the inclusions of these sets 
    occur in the opposite direction compared to the previous step.

  \item Introduce index sets $\set{I}_i := \set{K}_i \setminus \set{K}_{i-1}$ and $\set{J}_i := \set{L}_{i-1} \setminus \set{L}_{i}$
    for each $i \in \{1,2,\ldots,n\}$. Consider $X$ as a block matrix $X = (\indexmat{X}{\set{I}_i}{\set{J}_j})_{i,j=1}^n$ induced by the partitioning
    index sets $\{\set{I}_i\}_{i=1}^n$ and
    $\{\set{J}_j\}_{j=1}^n$.

  \item Choose the strictly upper triangular blocks $\indexmat{X}{\set{I}_i}{\set{J}_j}$ for $i < j$ arbitrarily. \label{step:free}

  \item At this point the remainder of the entries of a minimal $X$ are fixed by the choices made. For each $i =1,2,\ldots,n$ in sequence,
    set $\overline{\set{L}_i} := \set{L}_0 \setminus \set{L}_i = \bigcup_{j=1}^i \set{J}_j$ and assign the $i$th block row $\indexmat{X}{\set{I}_i}{\overline{\set{L}_i}}$
    to be the unique solution of the following minimal rank completion problem 
    \begin{equation*}
      \indexmat{X}{\set{I}_i}{\overline{\set{L}_i}}
      = \argmin_Y
      \rank\resizebox{.6\hsize}{!}{$\left[\arraycolsep=1.1pt\def\arraystretch{1.3}
          \begin{array}{c|c|ccc}
            A_{i1}(:,\overline{\set{L}_i}) & A_{i1}(:,\set{L}_{i}) & A_{i2} & \cdots & A_{ii} \\ 
            A_{(i+1)1}(:,\overline{\set{L}_i}) & A_{(i+1) 2}(:,\set{L}_{i}) & A_{(i+1)2} & \cdots & A_{(i+1)i} \\
            \vdots & \vdots & \vdots & \ddots    & \vdots \\
            A_{(n-1)1}(:,\overline{\set{L}_i}) & A_{(n-1)1}(:,\set{L}_{i}) & A_{(n-1)2} & \cdots & A_{(n-1)i} \\ \hline
            \indexmat{X}{\set{K}_{i-1}}{\overline{\set{L}_{i}}} &
                                                                    \indexmat{X}{\set{K}_{i-1}}{\set{L}_i} &                                                                                                                     A_{n2}(\set{K}_{i-1},:) & \cdots & A_{ni}(\set{K}_{i-1},:) \\
            \hline
            Y &
                \indexmat{X}{\set{I}_{i}}{\set{L}_i} &                                                                                                                       A_{n2}(\set{I}_{i},:) & \cdots & A_{ni}(\set{I}_{i},:) \\            
          \end{array} \right]$},
    \end{equation*} 
    as furnished by the unique completion lemma.\label{step:omrcp_lower}
  \end{enumerate}
  Some key features of this construction are illustrated graphically in Figure~\ref{fig:construction}.
\end{construction}

\begin{figure}[t]
  \centering
    \begin{subfigure}[b]{0.4\textwidth}
    \centering
    \includegraphics[width=\textwidth]{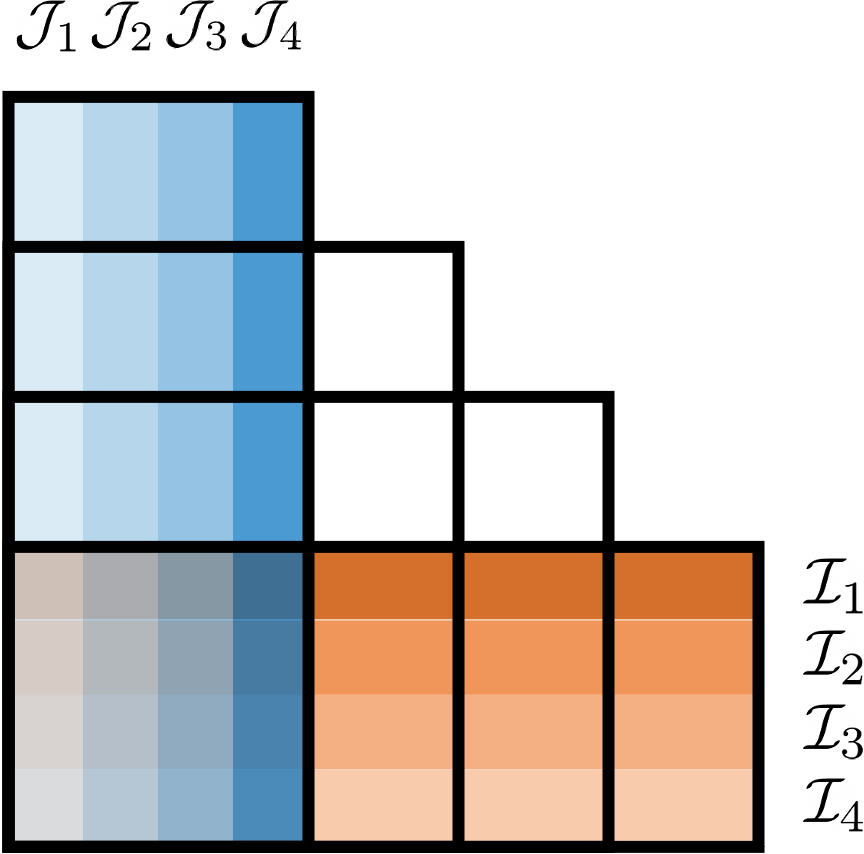}
    \caption{}\label{fig:4x4_example}
  \end{subfigure}
  \quad
  \begin{subfigure}[b]{0.4\textwidth}
    \centering
    \includegraphics[width=\textwidth]{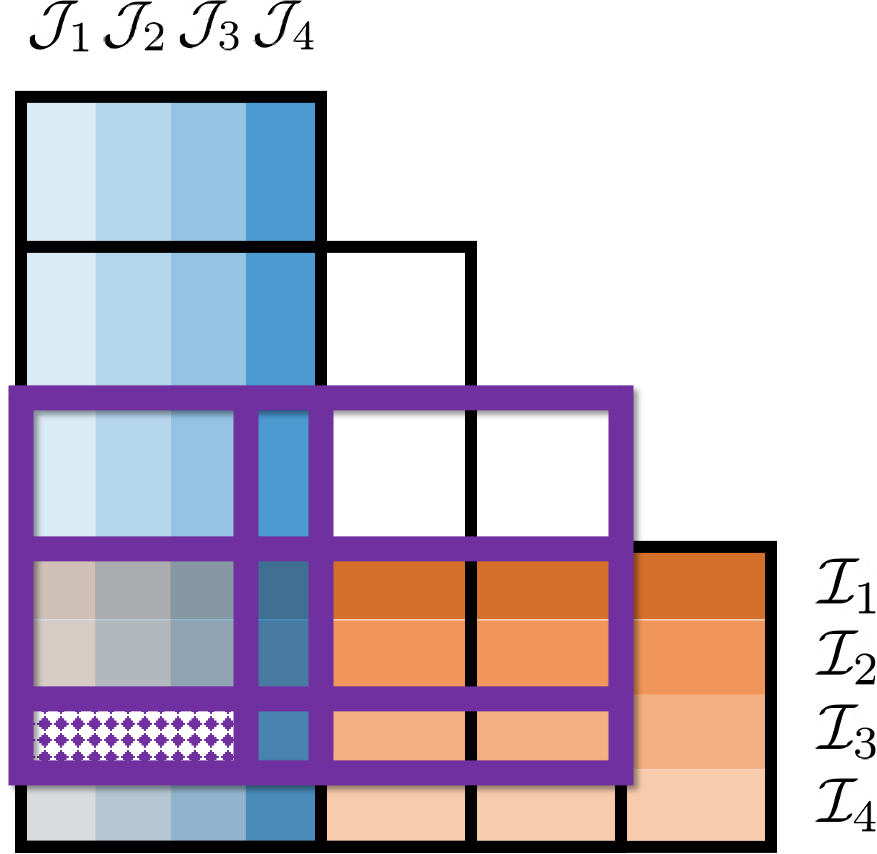}
    \caption{}\label{fig:4x4_lower}
  \end{subfigure}  
  \caption{Illustration of the construction of the overlapping block simultaneous rank-minimizing $X$ in the case $n = 4$: (\subref{fig:4x4_example}),
    partition of rows and columns of $X$ into index sets and (\subref{fig:4x4_lower}), determination of the lower triangular portion of
    block row $i = 3$ entry by solving a
    block $2\times 2$ minimal rank completion problem (shown in purple) partitioned as $3\times 3$ to satisfy the requirements of the the unique completion lemma (Lemma~\ref{lem:ucl}).}
  \label{fig:construction}
\end{figure}

We note that there are different orders for filling in the remaining portions of the
matrix as done in step~\ref{step:omrcp_lower}, but these all produce the same result
given the same choices of the free variables in step~\ref{step:free}. We prefer the order we
present in step~\ref{step:omrcp_lower} because it is easy to analyze.
Let us now present a more refined version of Theorem~\ref{thm:main_theorem}, which we shall then prove.

\begin{theorem} 
  The overlapping block minimal rank completion problem \eqref{eq:overlapping_block_mrcp} always
  possesses a solution and all such solutions are produced by
  Construction~\ref{cons:overlapping_block_mrcp}. In particular, the solution space of \eqref{eq:overlapping_block_mrcp} is an affine space of dimension
  \begin{equation*}
    \sum_{1\le i < j \le n} (\alpha_i - \alpha_{i-1})(\beta_{j-1} - \beta_j),
  \end{equation*}
  where
  \begin{align*}
    \alpha_i &= \rank \begin{bmatrix} A_{(i+1)2} & \cdots & A_{(i+1)(i+1)} \\
      \vdots & \ddots & \vdots \\
      A_{(n-1)2} & \cdots & A_{(n-1)(i+1)} \\
      A_{n2} & \cdots & A_{n(i+1)} \end{bmatrix} -
    \rank \begin{bmatrix} A_{(i+1)2} & \cdots & A_{(i+1)(i+1)} \\
      \vdots & \ddots & \vdots \\
      A_{(n-1)2} & \cdots & A_{(n-1)(i+1)}
    \end{bmatrix},& i&\in \{1,\ldots,n-1\};\\
    \beta_j &= \rank \begin{bmatrix} A_{j1} & A_{j2} & \cdots & A_{jj} \\
      \vdots & \vdots & \ddots & \vdots \\
      A_{(n-1)1} & A_{(n-1)2} & \cdots & A_{(n-1)j}
    \end{bmatrix} -
                                         \rank \begin{bmatrix} A_{j2} & \cdots & A_{jj} \\
                                           \vdots & \ddots & \vdots \\
                                           A_{(n-1)2} & \cdots & A_{(n-1)j}
                                         \end{bmatrix},&
                                                         j&\in \{2,\ldots,n\}.
  \end{align*}
\end{theorem}

\begin{proof} 
  Let us first check that the hypotheses of the unique completion lemma are indeed satisfied 
  when it is invoked in step~\ref{step:omrcp_lower} of Construction~\ref{cons:overlapping_block_mrcp}.
  For convenience, we denote the block entries of the matrix completion problem considered at iteration
  $i \in \{1,2,\ldots,n\}$ using notation consistent with the unique completion lemma:
  \begin{equation*}
    \threebythree{B_1}{C_{11}}{C_{12}}{B_2}{C_{21}}{C_{22}}{Y}{D_1}{D_2} =
    \left[\arraycolsep=1.1pt\def\arraystretch{1.3}
      \begin{array}{c|c|ccc}
        A_{i1}(:,\overline{\set{L}_i}) & A_{i1}(:,\set{L}_i) & A_{i2} & \cdots & A_{ii} \\ 
        A_{(i+1)1}(:,\overline{\set{L}_i}) & A_{(i+1) 1}(:\set{L}_i) & A_{(i+1)2} & \cdots & A_{(i+1)i} \\
        \vdots & \vdots & \vdots & \ddots    & \vdots \\
        A_{(n-1)1}(:,\overline{\set{L}_i}) & A_{(n-1)1}(:,\set{L}_i) & A_{(n-1)2} & \cdots & A_{(n-1)i} \\ \hline
        \indexmat{X}{\set{K}_{i-1}}{\overline{\set{L}_i}} & \indexmat{X}{\set{K}_{i-1}}{\set{L}_i} & A_{n2}(\set{K}_{i-1},:) & \cdots & A_{ni}(\set{K}_{i-1},:) \\ \hline
        Y & \indexmat{X}{\set{I}_i}{\set{L}_i} & A_{n2}(\set{I}_i,:) & \cdots & A_{ni}(\set{I}_i,:)
      \end{array}
    \right].
  \end{equation*}
  By construction of $\set{L}_i$ and $\set{K}_{i-1}$ (steps \ref{step:2x2_lrcp_cols} and
  \ref{step:2x2_lrcp_rows}), we have that $\Col \onebythree{B_1}{C_{11}}{C_{12}} = \Col \onebytwo{C_{11}}{C_{12}}$ and
  $\Row \threebyone{C_{12}}{C_{22}}{D_2} = \Row \twobyone{C_{12}}{C_{22}}$. It is an easy induction argument that the columns of $C_{11}$ and the
  rows of $C_{22}$ are linearly independent. 
  
  We now check that $\Row C_{12} \cap \Row C_{22} = \{0\}$. We prove by induction on $i$. The
  base cases $i = 1$ and $i=2$ are clear. Now suppose the claim holds for $i-1$ for some $i > 2$.
  From this,  it is clear that $\Row C_{12} \cap \Row C_{22}(\set{K}_{i-2},:) = \{0\}$.
  Since $\set{K}_{i-1}$ is a minimal extension
  of $\set{K}_{i-2}$, it follows that $\Row C_{12} \cap \Row C_{22} = \{0\}$. For if it
  were the case that $\Row C_{12} \cap \Row C_{22} \supsetneq \{0\}$, then $\set{K}_{i-1}$ could be made smaller by at least one index, contradicting minimality of $\set{K}_{i-1}$.
  An entirely analogous argument shows that $\Col C_{11} \cap \Col C_{12} = \{0\}$.
  Thus, the hypotheses of the unique completion lemma are satisfied in this case and there exists a unique rank-minimizing $\indexmat{X}{\set{I}_i}{\overline{\set{L}_i}}$
  depending in an affine way on the freely chosen portion of $X$.

  Now, let us establish inductively that this $X$ does indeed minimize the rank of all the blocks,
  starting with step $i=n$ as our base case. Throughout, denote $\overline{\set{K}_i} :=
  \set{K}_n \setminus \set{K}_i$. By the solution properties \ref{item:ucl_row}-\ref{item:ucl_col}
  provided by the unique completion lemma for step $i = n$ of the construction, we have that
  $\Col X \subseteq \Col \begin{bmatrix} A_{n2} & \cdots & A_{nn} \end{bmatrix}$ and  
  \begin{gather*}
      \Row \begin{bmatrix} \indexmat{X}{\overline{\set{K}_{n-1}}}{:}
    & A_{n2}(\overline{\set{K}_{n-1}},:) & \cdots & A_{nn}(\overline{\set{K}_{n-1}},:) \end{bmatrix} \\
    \subseteq \Row \begin{bmatrix}
    \indexmat{X}{\set{K}_{n-1}}{:}
    & A_{n2}(\set{K}_{n-1},:) & \cdots & A_{nn}(\set{K}_{n-1},:) \end{bmatrix}.
  \end{gather*}
  Thus, by the first of these properties,
  the rank of the $n$th block is given by
  \begin{equation*}
    \rank \begin{bmatrix} X & A_{n2} & \cdots & A_{nn} \end{bmatrix} = \rank \begin{bmatrix} A_{n2} & \cdots & A_{nn} \end{bmatrix},
  \end{equation*}
  which is optimal.

  Next, seeking to show that $X$ minimizes the rank of the $i$th block, inductively assume that $X$ minimizes the rank
  of blocks $i+1,\ldots,n$ and assume that 
  \begin{equation} \label{eq:inductive_row}
    \Row \begin{bmatrix} \indexmat{X}{\overline{\set{K}_i}}{:}
      & A_{n2}(\overline{\set{K}_i},:) & \cdots & A_{n(i+1)}(\overline{\set{K}_i},:) \end{bmatrix} \subseteq \Row \begin{bmatrix}
      \indexmat{X}{\set{K}_i}{:}
      & A_{n2}(\set{K}_i,:) & \cdots & A_{n(i+1)}(\set{K}_i,:) \end{bmatrix}.
  \end{equation}
  By the optimality properties \ref{item:ucl_row}-\ref{item:ucl_col} of the unique completion lemma, we have 
  \begin{align}
    \Row \begin{bmatrix} \indexmat{X}{\set{I}_i}{:}
      & A_{n2}(\set{I}_i,:) & \cdots & A_{ni}(\set{I}_i,:) \end{bmatrix} &\subseteq \Row \begin{bmatrix}
      \indexmat{X}{\set{K}_{i-1}}{:}
      & A_{n2}(\set{K}_{i-1},:) & \cdots & A_{ni}(\set{K}_{i-1},:) \end{bmatrix}, \label{eq:row_property} \\
    \Col \begin{bmatrix} A_{i1}(:,\overline{\set{L}_i}) \\ \vdots \\ A_{(n-1)1}(:,\overline{\set{L}_i}) \\ \indexmat{X}{\set{K}_i}{\overline{\set{L}_i}}
    \end{bmatrix}&\subseteq \Col \begin{bmatrix} A_{i1}(:,\set{L}_i) & A_{i2} & \cdots & A_{ii} \\ \vdots & \vdots & \ddots & \vdots \\
      A_{(n-1)1}(:,\set{L}_i) & A_{(n-1)2} & \cdots & A_{(n-1)i} \\
      \indexmat{X}{\set{K}_i}{\set{L}_i} & A_{n2}(\set{K}_i,:) & \cdots & A_{ni}(\set{K}_i,:) \end{bmatrix}. \label{eq:col_property} 
  \end{align}
  Combining \eqref{eq:inductive_row} and \eqref{eq:row_property} establishes the inductive hypothesis \eqref{eq:inductive_row} for block $i$.
  Together \eqref{eq:inductive_row}, \eqref{eq:row_property}, and \eqref{eq:col_property} show that
  \begin{equation*}
    \rank \begin{bmatrix} A_{i1} & A_{i2} & \cdots & A_{ii} \\
      \vdots & \vdots & \ddots & \vdots \\
      A_{(n-1)1} & A_{(n-1)2} & \cdots & A_{(n-1)i} \\
      X & A_{n2} & \cdots & A_{ni} \end{bmatrix} =
    \rank \begin{bmatrix} A_{i1}(:,\set{L}_i) & A_{i2} & \cdots & A_{ii} \\
      \vdots & \vdots & \ddots & \vdots \\
      A_{(n-1)1}(:,\set{L}_i) & A_{(n-1)2} & \cdots & A_{(n-1)i} \\
      \indexmat{X}{\set{K}_{i-1}}{\set{L}_i} & A_{n2}(\set{K}_{i-1},:) & \cdots & A_{ni}(\set{K}_{i-1},:) \end{bmatrix},
  \end{equation*}
  which saturates the optimal rank bound \eqref{eq:2x2_lower_bound} by the minimality of $\set{K}_{i-1}$ and $\set{L}_i$.
  The fact that no simultaneously minimizing $X$ exists other than as produced by this construction follows from the uniqueness
  of the solution furnished by the unique completion lemma at each step of the construction. The affine dimension of the solution space
  is immediate from counting the freely chosen subblocks of $X$.
\end{proof}    

\section{Discussion and Conclusions} \label{sec:discussion_and_conclusions}

Let us first remark on issues of numerical implementation for solving this problem using finite-precision arithmetic where the base
field is either the real or complex numbers. The construction we used to prove Theorem~\ref{thm:main_theorem} involves selecting
minimal sets of rows or columns of a matrix which span a certain row or column space. This is often undesirable in numerical computations
as the row and column sets can be ill-conditioned. Fortunately, this is not an essential feature of this construction. These row and column
subsets can be replaced by orthonormalized bases for these same row and column spaces and the free variables in the solution correspond to the
strictly block upper triangular portion of $X$ after an appropriate change of basis. When phrased in this way, the algorithm requires
the computation of orthonormal bases for intersection of subspaces as its primitive operation, which can be computed
using SVD-based algorithms \cite{bjorck1973numerical}. We note that this problem is inherently numerically unstable
and the solution set of the problem can be prone to dramatic changes based on small thresholding decisions about the numerical rank of a matrix. 

Let us also remark on an interesting consequence of our analysis which may have
computational importance.
\begin{corollary}
If any subcollection of the blocks \eqref{eq:hankel_blocks} possesses a \emph{unique}
rank-minimizing $X$, then $X$ is guaranteed to minimize the rank of all the other blocks.
\end{corollary}
The proof is immediate from Theorem~\ref{thm:main_theorem}: if this unique $X$ failed to
minimize the rank of all the other blocks, then there would be no joint minimizer at all,
contradicting the theorem. This observation has the following consequence.
In the class of examples where we expect the overlapping block minimal rank completion
problem to have a unique solution, an effective strategy might be to pick a block
at random and check whether it has a unique rank-minimizing solution or if the dimension of the solution space is a small number. In this case, it is
possible we will find the rank-minimizing choice or a nearly rank-minimizing choice much more efficiently than if we performed the entire construction. We found empirically
that this was an effective strategy for examples we considered in \cite{CEG19}.

There are several extensions of this work that may be possible. One natural question is to see if this analysis can be generalized to more complicated overlapping minimal
rank completion problems, some of which emerge in finding rank-structural representations for more complicated graph structures in our motivating application \cite{CEG19}.
 We consider it an interesting open problem of how large a class of general overlapping block minimal rank
completion problems possess simultaneous minimizers which can
be computed tractably, in view of hardness results such as \cite{buss_computational_1999}. 

\section*{Acknowledgements}

We thank the anonymous reviewer for suggesting a significantly shorter and more revealing proof of Theorem~\ref{thm:main_theorem} than the one we originally discovered,
from which the proof presented in this article has been adapted. 

\bibliographystyle{plain}
\bibliography{references}

\appendix

\section{The Form of \texorpdfstring{$C^{-1}$}{C inverse}} \label{app:Cinvder}

Here, we derive the block structure of the inverse of the matrix $C$ as needed in the proof of Lemma~\ref{lem:ucl}. Let
\begin{equation*}
    C^{-1} = \twobytwo{Y_{11}}{Y_{12}}{Y_{21}}{Y_{22}}
\end{equation*}
be a partition of $C$ such such that the products $CC^{-1}$ and $C^{-1}C$ are conformal.
Then since $C^{-1}C = I$, we get the following equations:
\begin{subequations} \label{eq:C}
\begin{align}
    Y_{11} C_{11} + Y_{12} C_{21} &= I, \label{eq:C1} \\
    Y_{11} C_{12} + Y_{12} C_{22} &= 0, \label{eq:C2}\\
    Y_{21} C_{11} + Y_{22} C_{21} &= 0, \label{eq:C3}\\
    Y_{21} C_{12} + Y_{22} C_{22} &= I. \label{eq:C4}
\end{align}
\end{subequations}
By \eqref{eq:C2} and the hypothesis that $\Row C_{12} \cap \Row C_{22} = \{0\}$, we conclude
that $Y_{12} = 0$. Thus, plugging into \eqref{eq:C1} and \eqref{eq:C2}, we discover
\begin{equation*}
    Y_{11} \onebytwo{C_{11}}{C_{12}} = \onebytwo{I}{0} \implies Y_{11} = \onebytwo{I}{0}\onebytwo{C_{11}}{C_{12}}^\rinv.
\end{equation*}
But since $CC^{-1} = I$, we have
\begin{equation*}
    \twobyone{C_{21}}{C_{22}}Y_{22} = \twobyone{0}{I} \implies Y_{22} = \twobyone{C_{21}}{C_{22}}^\linv \twobyone{0}{I}.
\end{equation*}
Then by \eqref{eq:C3} and \eqref{eq:C4},
\begin{equation*}
    Y_{21} \onebytwo{C_{11}}{C_{12}} + Y_{22} \onebytwo{C_{21}}{C_{22}} = \onebytwo{0}{I}
    \implies Y_{21} = \left(\onebytwo{0}{I} - Y_{22} \onebytwo{C_{21}}{C_{22}}\right)\onebytwo{C_{11}}{C_{12}}^\rinv.
\end{equation*}
This completes the claimed descriptions of the block structure of $C^{-1}$.

\end{document}